\documentclass[12pt]{amsart}
\usepackage{graphicx,dbnsymb,amssymb,floatflt,color,epic,eepic,pb-diagram}
\usepackage{picins} 
\usepackage{amsmath}
\usepackage[all]{xy}

\usepackage{amsthm}
\theoremstyle{plain}
\newtheorem*{theorem*}{Main Theorem}

\newtheorem{proposition}{Proposition}[section]

\theoremstyle{definition}
\newtheorem{definition}[proposition]{Definition}

\newtheorem{example}[proposition]{Example}

\newtheorem{remark}[proposition]{Remark}

\newlength{\standardunitlength}
\catcode`\@=11
\long\def\@makecaption#1#2{%
    \vskip 10pt
    \setbox\@tempboxa\hbox{%\ifvoid\tinybox\else\box\tinybox\fi
      \small\sf{\bfcaptionfont #1. }\ignorespaces #2}%
    \ifdim \wd\@tempboxa >\captionwidth {%
        \rightskip=\@captionmargin\leftskip=\@captionmargin
        \unhbox\@tempboxa\par}%
      \else
        \hbox to\hsize{\hfil\box\@tempboxa\hfil}%
    \fi}
\font\bfcaptionfont=cmssbx10 scaled \magstephalf
\newdimen\@captionmargin\@captionmargin=2\parindent
\newdimen\captionwidth\captionwidth=\hsize
\catcode`\@=12

\def\qed{{\hfill\text{$\Box$}}}

\newlength{\globalparindent}
\def\calA{{\mathcal A}}
\def\calC{{\mathcal C}}

\def\calO{{\mathcal O}}
\def\calP{{\mathcal P}}

\def\calSo{{\mathcal S}_o}

\def\calD{{\mathcal D}}

\def\calMko{{\mathcal M}_k^{(o)}}

\newcommand{\Cobo}{{\mathcal Cob}^3_o}

\newcommand{\Cobdl}{{\mathcal Cob}_{\bullet/l}}

\newcommand{\Kob}{\operatorname{Kob}}

\newcommand{\Kobo}{\operatorname{Kob}_o}

\newcommand{\Koboh}{{\operatorname{Kob}_{o/h}}}

\newcommand{\Kom}{\operatorname{Kom}}
\newcommand{\Komh}{\operatorname{Kom}_{/h}}

\newcommand{\Mat}{\operatorname{Mat}}

\renewcommand{\qed}{~\hfill$\square$}

\begin{document}
\newdimen\captionwidth\captionwidth=\hsize

\title{A binary operation on the class of coherently diagonal complexes}

\author{Hernando Burgos Soto}
\address{ George Brown College\\Toronto, ON, M1C 5J9\\Canada} \email{hburgos@georgebrown.ca}
\urladdr{http://individual.utoronto.ca/hernandoburgos/}

\date{
  First edition: March.{} 22, 2010.
  This edition: January 23, 2011.
}


%\thanks{This work was partially supported by The Universidad del Norte in Barranquilla, Colombia and The Universidad Nacional in Bogota, Colombia.}

\begin{abstract}
  We use mathematical induction to prove that the horizontal composition in the class of coherently diagonal complexes is indeed a binary operation. That is to say, the embedding of two coherently diagonal complexes in an alternating planar diagram produces a coherently diagonal complex.
\end{abstract}

%\dedicatory{To , who  $\slashoverback\mapsto
%A\,\smoothing+A^{-1}\hsmoothing$.}

\maketitle

\tableofcontents

\section{Introduction} \label{sec:intro}

\indent In \cite{Bur}, we define the category $\Cobo$ whose objects are {\it oriented smoothing}, and whose morphisms are {\it oriented
cobordisms}.
This orientation in the smoothings was at that moment utilized in order 
to define a parameter belonging to $\frac{1}{2}\mathbb{Z}$ associated to the smoothing, and then generalize a Thistlethwaite result for the Jones polynomial stated in \cite{Thi}.\\
\indent If $\sigma$ denotes an oriented smoothing, the parameter associated to  $\sigma$, is called its {\it
rotation number} and is denoted by $R(\sigma)$. Specifically, for degree-shifted smoothings
$\sigma\{q\}$ we define $R(\sigma\{q\}): = R(\sigma) + q$. We further use this {\it
degree-shifted rotation number} to define a special class of chain
complexes in $\Kom(\Mat(\Cobo))$, of the form
\[\Omega:\qquad \cdots \longrightarrow \left[\sigma^r_j\right]_j
\longrightarrow \left[\sigma^{r+1}_j\right]_j \longrightarrow \cdots ,\]
which satisfies that for all degree-shifted smoothings $\sigma^r_j\{q\}$,  $2r-R(\sigma^r_j\{q\})$ is a
constant that we call {\it rotation constant} of the complex. In other words, twice the homological
degrees and the degree-shifted rotation numbers of the smoothings always lie along a
single diagonal.
We call this type of chain complexes {\it diagonal complexes}. Furthermore, a {\it coherently diagonal} complex is a diagonal complex whose partial closure
is also diagonal. Complexes of this type are the objects in our main theorem \\
\begin{theorem*}\label{Theorem:MainTheo1} Let $\Omega_1$ and $\Omega_2$ be coherently diagonal complexes, and let $D$ a binary operator of the alternating planar algebra. Then $D(\Omega_1,\Omega_1)$ is a diagonal complex
diagrams).
\end{theorem*}
 
 \indent The work is organized as follows. In section
\ref{sec:QuickReview}, we review the concept of bounded chain complex and present two additional tools for the proof of theorem \ref{Theorem:MainTheo1}. These tools are propositions \ref{lem:TriangulOfBlocks} and \ref{prop:GaussEliminitionInTriangular}. Section \ref{sec:Alternating} is devoted to introduce the category $\Cobo$ an give a quick review of some concepts related to
 alternating planar algebras. In particular we review the concepts of rotation number, alternating planar diagram, associated rotation number, and basic operators.\\
\indent Section \ref{sec:On-Diagonal}  introduces the concepts of diagonal complexes, coherently diagonal complexes, and their partial closures. We state here some results about the complexes obtained when a basic operator is applied to alternating elements, leading to the prove in section \ref{sec:Theorem1} of Theorem \ref{Theorem:MainTheo1}.\\
%\section{Acknowledgement} I wish to thank D. Bar-Natan, for many
%helpful conversations we had at the University of Toronto and for
%allowing me to use some figures from \cite{Bar1,Bar2}. I would also
%like to thank N. Martin for their comments and
%suggestions.
\section{Bounded chain complexes}\label{sec:QuickReview}

\indent For our purposes, it will be useful to recall here the concept of {\it bounded} chain complex. See \cite{Wei}. A chain complex
  \[\Omega:\qquad \xymatrix{  \cdots \Omega^r
\ar[r]^{d^r} & \Omega^{r+1}
  \cdots }\]
is called {\it bounded} if almost all the $\Omega^r$ are zero. If $\Omega^{r_o} \neq 0$, $\Omega^{r_M}\neq 0$ and $\Omega^{r}=0$ unless $r_0 \leq r \leq r_M$, we say that $(\Omega,d)$ has {\it amplitude} in $[r_0, r_M]$. \\

\indent It is well known that if
$(\Omega_i, d_i)\in\Kob(k_i)$ are complexes,  the complex
$(\Omega,d)=D(\Omega_1,\dots,\Omega_n)$ is defined by
\begin{equation} \label{eq:KobPA}
\begin{split}
    \Omega^r & :=
    \bigoplus_{r=r_1+\dots+r_n} D(\Omega_1^{r_1},\dots,\Omega_n^{r_n})
  \\
    \left.d\right|_{D(\Omega_1^{r_1},\dots,\Omega_n^{r_n})} & :=
    \sum_{i=1}^n (-1)^{\sum_{j<i}r_j}
      D(I_{\Omega_1^{r_1}},\dots,d_i,\dots,I_{\Omega_n^{r_n}}),
\end{split}
\end{equation}

 \begin{proposition}\label{lem:TriangulOfBlocks}Let $(\Psi,e)$ and $(\Phi,f)$ be chain complexes in $\Kob$; let $(\Phi,f)$ be a bounded complex in $[t_0, t_M]$ with $t_M-t_0=N-1$, and let $D$ be a 2-input planar arc diagram in which $D(\Psi,\Phi)$ makes sense. Then  $D(\Psi,\Phi)$ is homotopy equivalent to a chain complex $(\Omega,d)$ with the following properties:
\begin{enumerate}
\item Every vector $\Omega^r$ is of the form \[\Omega^r=
\bigoplus_{\begin{array}{c}
             t_0 \leq t \leq t_M \\
             s=r-t
           \end{array}
}D(\Psi^{s},\Phi^{t})\] can be regarded as a block column matrix $\left(
                                                                                       \begin{array}{c}
                                                                                         \Omega_1^r \\
                                                                                         \vdots \\
                                                                                         \Omega_{N}^r \\
                                                                                       \end{array}
                                                                                     \right)$ in which each block $\Omega_t^r$ is given by \[\Omega_t^r=D(\Psi^s,\Phi^t)\]
                                                                                     
 \item The differential matrices $d^r$ can be seen as ($N\times N$)-block  lower triangular matrices with blocks $d_{ij}^r:\Omega_j^r \longrightarrow \Omega_i^{r+1}$, $i,j \in {1,\ldots,N}$. That is to say, $d^r_{ij}$ is the summand in $d|_{D(\Psi, \Phi)}$ that has domain $\Omega_j^r$ and image  $\Omega_i^{r+1}$
\end{enumerate}
\end{proposition}
\begin{proof}
The two statements follows immediately from the definition of $D(\Psi,\Phi)$, equations (\ref{eq:KobPA}).  Obviously, if $s\leq s_0$ or $s \geq s_M$ we consider $\Psi^s=0$.\\
\indent To prove that by providing this order to the elements in the complex we obtain ($N\times N$)-block  lower triangular matrices in the differentials,  we see that given $r=s+t$, the matrix $d^r$ is defined by the second of the equations (\ref{eq:KobPA}), and is given by \begin{equation}\label{eq:diffbasic1}\left.d\right|_{D(\Psi^{s},\Phi^{t})}=
D(e,I_{\Phi^{t}})+(-1)^{s}D(I_{\Psi^{s}},f).\end{equation} Since $d_{ii}$ are of the form  $d_{ii}^r:D(\Psi^s,\Phi^i) \longrightarrow D(\Psi^{s+1},\Phi^i)$, the summands in  $D(e,I_{\Phi^{t}})$ are clearly concentrated in the blocks located in the diagonal. Any other block is a morphism of the form $d_{ij}^r:D(\Psi^s,\Phi^j) \longrightarrow D(\Psi^{s},\Phi^i)$ with $i \neq j$ which are part of $\pm D(I_{\Psi^s},f)$ in the right side of equation (\ref{eq:diffbasic1}).  Furthemore, if $i<j$, $f_{ij}:\Phi^j\rightarrow \Phi^i$ is the zero cobordism.
\qed \end{proof}

\parpic[r]{$\begin{array}{c}\includegraphics[width=1.5in]%
{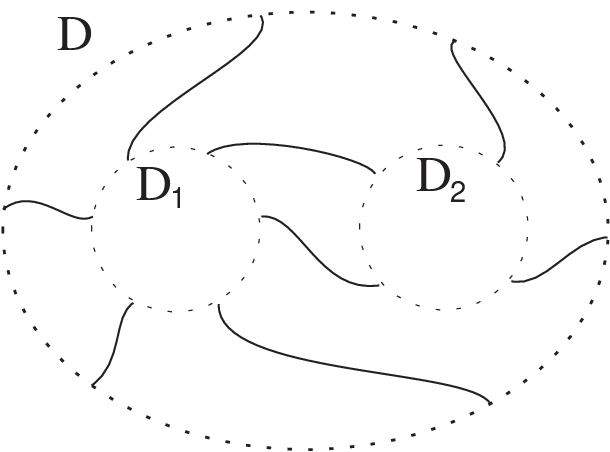}\end{array}$}We illustrate the previous proposition with an example. Let $D$ be the binary operator defined from the planar arc diagram of the right. If we place the complex

 \begin{equation*} \begin{array}{c} \Psi = \end{array} \hspace{.1cm} \begin{array}{c}\includegraphics[scale=0.5]%
{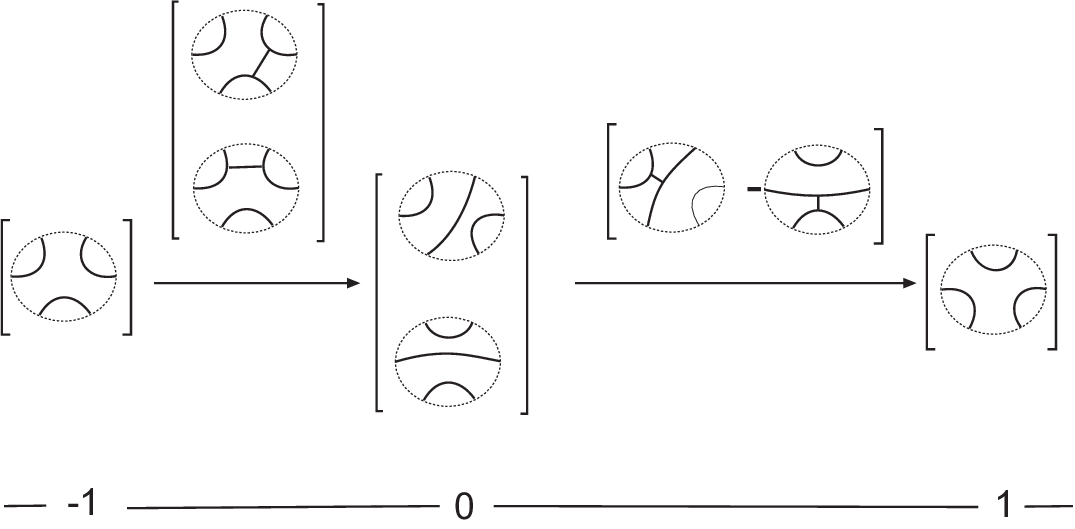}\end{array} \end{equation*}
in the first entry of $D$ and \begin{equation*} \begin{array}{c}\Phi = \end{array} \hspace{.1cm} \begin{array}{c}\includegraphics[scale=0.5]%
{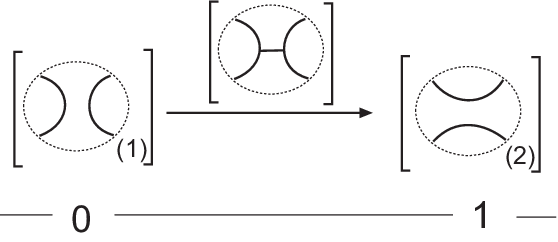}\end{array} \end{equation*}
in the second entry. Once we have embedded these complexes in $D$, we obtain a new complex:\\
\begin{equation*}\begin{array}{c}D(\Psi,\Phi)=  \end{array} \begin{array}{c}\includegraphics[scale=0.6]%
{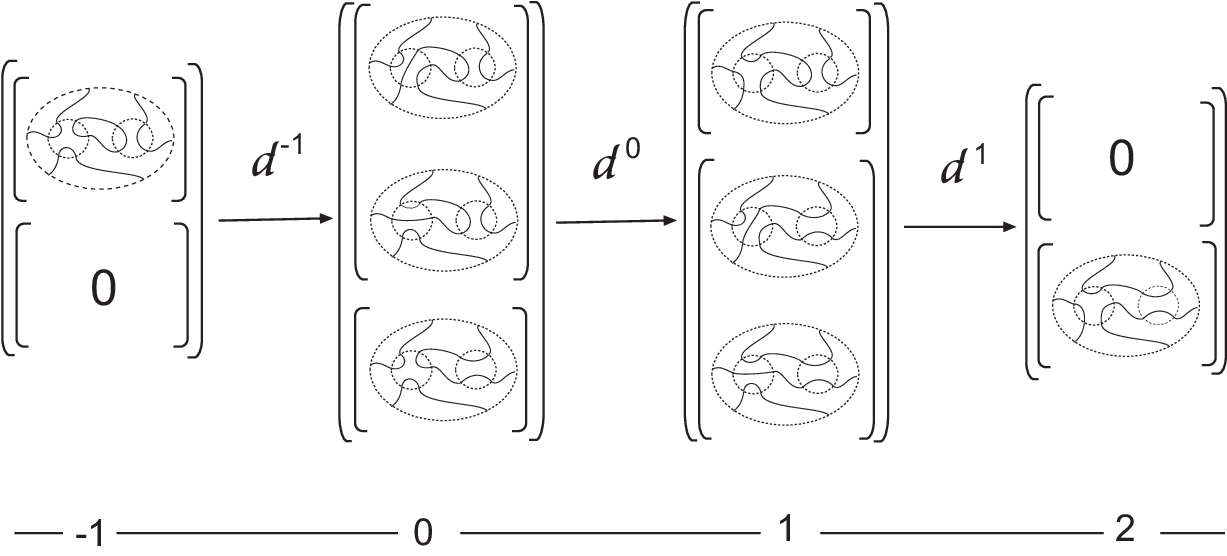}\end{array}\end{equation*}
The differentials in this complex can be seen as block-lower-triangular matrices, as they are displayed in Figure \ref{Fig:differentials}.
\begin{figure}[hbt] \centering
\begin{tabular}{c}\begin{tabular}{cc}\begin{tabular}{c} $d^{-1}=$ \end{tabular}\begin{tabular}{c}\includegraphics[scale=0.7]%
{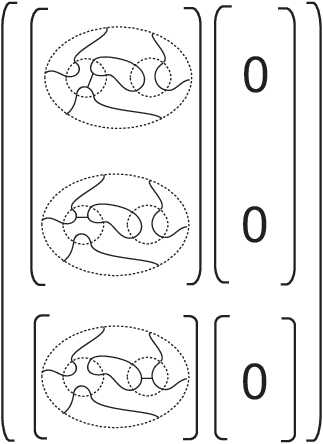}\end{tabular} & \begin{tabular}{c} $d^{0}=$ \end{tabular}\begin{tabular}{c}\includegraphics[scale=0.7]%
{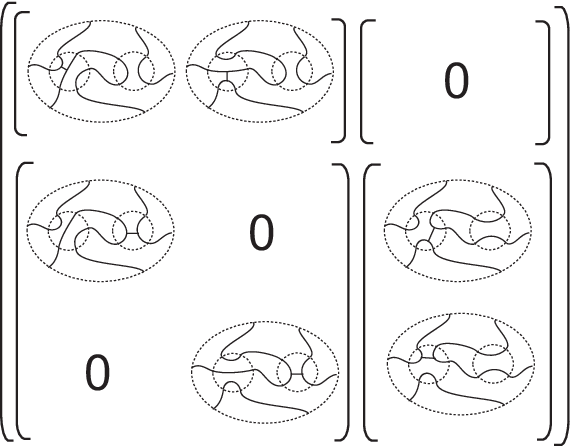}\end{tabular} \end{tabular} \\
\vspace{.1cm} \\
\begin{tabular}{c} $d^{1}=$ \end{tabular}\begin{tabular}{c}\includegraphics[scale=0.7]%
{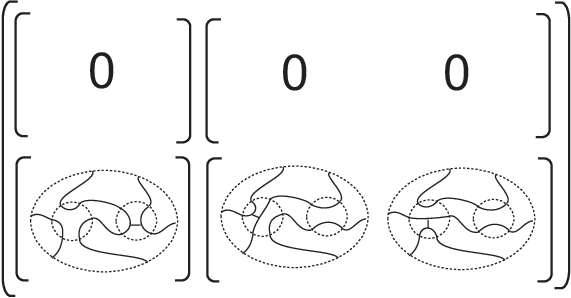}\end{tabular} \end{tabular}
\caption{The differentials in the complex $D(\Psi,\Phi)$. The blocks in the diagonal are the differentials  $d^r_{ii}:D(\Psi^{s},\Phi^i)\rightarrow D(\Psi^{s+1},\Phi^i)$. the elements in the blocks below the diagonals could have a sign shift. The blocks above the diagonal are blocks of zeros.
}\label{Fig:differentials}
\end{figure}
\begin{remark}\label{re:ComplexSmoothingComposition} The blocks $\Omega^r_i=D(\Psi^s,\Phi^i)$ in $\Omega^r$, and the blocks $d_{ii}^r:D(\Psi^s,\Phi^i) \longrightarrow D(\Psi^{s+1},\Phi^i)$ in the diagonal of $d^r$, determine the $N$ complexes \[D(\Psi,\Phi^i)=\xymatrix{  \cdots \Omega_{ii}^r
\ar[r]^{d_i^r} & \Omega_{ii}^{r+1}
  \cdots } \, \text{ } \, i \in {1,\dots , N} \] \qed
\end{remark}

\begin{proposition}\label{prop:GaussEliminitionInTriangular}Assume that the three differential matrices in the four term complex segment in $\Mat(\calC)$ of lemma \ref{lem:GaussianElimination} are block-lower-triangular matrices. After applying gauss elimination, the resulting three differential matrices in the four term complex segment are also block-lower-triangular matrices. Furthermore, the lowest right block of the three initial differential matrices remained unchanged after the Gauss elimination.
\end{proposition}

\begin{proof} It is clear that if $\left(
                                     \begin{array}{c}
                                       \alpha \\
                                       \beta\\
                                     \end{array}
                                   \right)
$, $\left(
      \begin{array}{cc}
        \phi & \delta \\
        \gamma & \epsilon \\
      \end{array}
    \right)
$, and $\left(
          \begin{array}{cc}
            \mu & \nu \\
          \end{array}
        \right)
$ are block-lower-triangular matrices, so they are $\beta$, $\epsilon$, and $\nu$. Therefore, it is clear that after Gauss elimination, the first and the third of the differential matrices in the form term complex are block-lower-triangular matrices with the same initial lowest-right block.\\
\indent To prove that the same happens with the second block, we observe that if $\left(
      \begin{array}{cc}
        \phi & \delta \\
        \gamma & \epsilon \\
      \end{array}
    \right)
$ is a block-lower-triangular matrices then $\left(
      \begin{array}{cc}
        \phi & \delta \\
        \gamma & \epsilon \\
      \end{array}
    \right)
=\left(
      \begin{array}{ccc}
        \phi & \delta_1 & 0 \\
        \gamma_1 & \epsilon_1 & 0 \\
\gamma_2 & \epsilon_2 & \epsilon_3 \\
      \end{array}
    \right)
$; where $\delta= \left(
                    \begin{array}{cc}
                      \delta_1 & 0 \\
                    \end{array}
                  \right)$, $\gamma=\left(
                                      \begin{array}{c}
                                        \gamma_1 \\
                                        \gamma_2 \\
                                      \end{array}
                                    \right)$, and  $\epsilon=\left(
                                                              \begin{array}{cc}
                                                                \epsilon_1 & 0 \\
                                                               \epsilon_2 & \epsilon_3 \\
                                                              \end{array}
                                                            \right)$. Each 0 in the previous matrices is actually a block of zeros.\\
\indent An immediate consequence of the previous paragraph is that the second differential matrix in the four term complex segment  is given by \[ \epsilon - \gamma \phi^{-1} \delta =\left(
                                      \begin{array}{cc}
                                        \epsilon_1 - \gamma_1 \phi^{-1} \delta_1 & 0 \\
                                        \epsilon_2 - \gamma_2 \phi^{-1} \delta_1 & \epsilon_3 \\
                                      \end{array}
                                    \right).
 \]
This completes the proof. \qed
\end{proof}

%%%%%%%%%%%%%%%%%%%%%%%%%%%%%%%%%%%%%%%%%%%%%%%%%%%%
%%%%%%%%%New section
%%%%%%%%%%%%%%%%%%%%%%%%%%%%%%%%%%%%%%%%%%%%%%%%%%%%

\section{The category $\Kobo$ and alternating planar algebras}  \label{sec:Alternating}
 We introduce an alternating orientation in the objects of $\Cobdl^3(k)$. This orientation induces an orientation in the cobordisms of this category. These oriented $k$-strand smoothings and cobordisms form the objects and morphisms in
a new category. The composition between cobordisms in
this oriented category is defined in the standard way, and it is regarded as a
graded category, in the sense of \cite[Section 6]{Bar1}. We subject out the
cobordisms in this oriented category to the relations in (\ref{eq:LocalRelations})
and denote it as $\Cobo(k)$. Now we can follow \cite{Bar1} and define sequentially the categories, $\Mat(\Cobo(k))$, $\Kom(\Mat(\Cobo(k)))$ and  $\Komh(\Mat(\Cobo(k)))$. This last two categories are what we denote $\Kobo(k)$, and $\Koboh$. As usual, we use $\Kobo$, and $\Koboh$, to denote $\bigcup_k \Kobo(k)$ and $\bigcup_k \Koboh(k)$ respectively.\\
\indent \parpic[r]{\includegraphics[scale=.45]%
{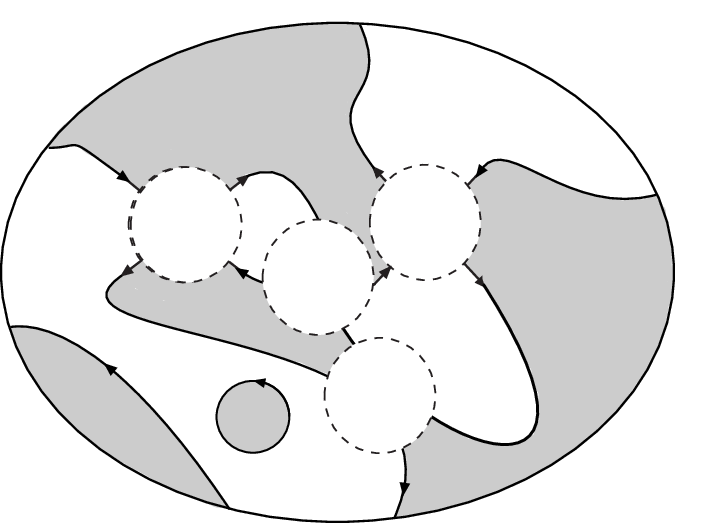}} We denote the class of oriented smoothings as $\calSo$. An alternatively oriented $d$-input
planar diagram, see \cite{Bur}, provides a good tool for the
horizontal composition of objects in $\calSo$, $\Cobo$, $\Mat(\Cobo(k))$, $\Kobo$, and $\Koboh$. The orientation in the diagrams can be provided as in the figure at the right. For making this text a little more self-contained, we are going to recall briefly some concepts presented before in \cite{Bur}. Given oriented smoothings
$\sigma_1,...,\sigma_d$, a suitable alternating $d$-input planar diagram $D$ to
compose them has the property that the $i$-th input disc has as many
boundary points as $\sigma_i$. Moreover placing $\sigma_i$ in the $i$-th input
disc, The orientation (the coloring) of $\sigma_i$ and $D$ match.\\

\indent Given an open strand $\alpha$ of an alternating oriented smoothing $\sigma$, possibly with loops, enumerate the boundary points of $\sigma$ in such a way that $\alpha$ can be denoted by $(0,i)$. The {\it rotation number} of $(0,i)$, $R(\alpha)$, is
$\frac{i-k}{2k}$. If $\alpha$ is a loop, $R(\alpha)=1$ if $\alpha$ is oriented
counterclockwise, and $R(\alpha)=-1$ if $\alpha$ is oriented clockwise. The
rotation number of $\sigma$ is the sum of the
rotation numbers of its strings. See figure \ref{Fig:rotate}
 \begin{figure}[th] \centering
\begin{tabular}{ccc}
\begin{tabular}{c} \includegraphics[scale=1.8]%
{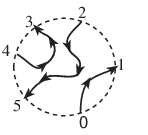} \end{tabular} & \hspace{1cm} & \begin{tabular}{c}
{\Large$R(\alpha)=\frac{1-3}{6}$}\end{tabular} \end{tabular}%
\caption{$\alpha =(0,1)$, $R(\alpha)=-\frac{2}{6}=-\frac{1}{3}$. The rotation number of
the complete resolution is 0}\label{Fig:rotate}
\end{figure}
 We are going to use this alternating diagrams to compute non-split alternating tangles, and we want to preserve the non-split property of the tangle. Hence, it will be better if we use $d$-input type $\calA$ diagrams.
\parpic[r]{\includegraphics[scale=.6]%
{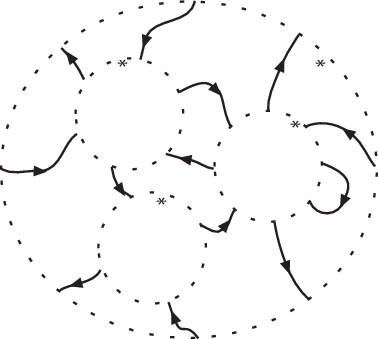}}A $d$-input type-$\mathcal{A}$ diagram  has an even number of
strings ending in each of its boundary components, and every string
that begins in the external boundary ends in a boundary of an
internal disk. We can classify the strings as:
{\it curls}, if they have its ends in the same input disc; {\it interconnecting arcs}, if its ends are in
different input discs, and {\it boundary arcs}, if they have one end in an input disc and the other in the external boundary
of the output disc. The arcs and the boundaries of the discs divide the surface of the diagram into disjoint regions. Some arcs and regions will be useful in the following definitions and propositions.

\begin{definition}We assign the following numbers to every $d$-input planar diagram $D$:
\begin{itemize}
\item $i_D$: number of interconnecting arcs and curls, i.e., the number of non-boundary arcs.
\item $w_D$: number of negative internal regions. That is, in the checkerboard coloring, the white regions
whose boundary does not meet the external boundary of $D$.
\item $R_D$: the rotation associated number, which is given by the
formula \[ R_D=\frac{1}{2}(1+i_D-d)-w_D\]
\end{itemize}.
\end{definition}
\begin{proposition}\label{prop:AsoRotNumber}
Given the smoothings $\sigma_1,...,\sigma_d$ and a suitable d-input planar
diagram $D$, where every smoothing can be placed, the rotation
number of $D(\sigma_1,...,\sigma_d)$ is:
\begin{equation}\label{eq:AsocRotNumber}
R(D(\sigma_1,...,\sigma_d))=R_D+\sum_{i=1}^dR(\sigma_i)\end{equation}
\end{proposition}
\begin{definition} An alternating planar algebra is a triplet $\{\calP,\calD,\calO\}$ in which $\calP$, $\calD$, and $\calO$ have the same properties as in the definition of a planar algebra but with the collection $\calD$ containing only $\calA$-type planar diagrams.
\end{definition}
Diagrams with only one or two input discs deserves special attention. Operators defined from diagram like these are very important for our purposes since some of them are considered as the generators of the entire collection of operators in a connected alternating planar algebra.
\begin{figure}[th] \centering
\includegraphics[scale=.6]%
{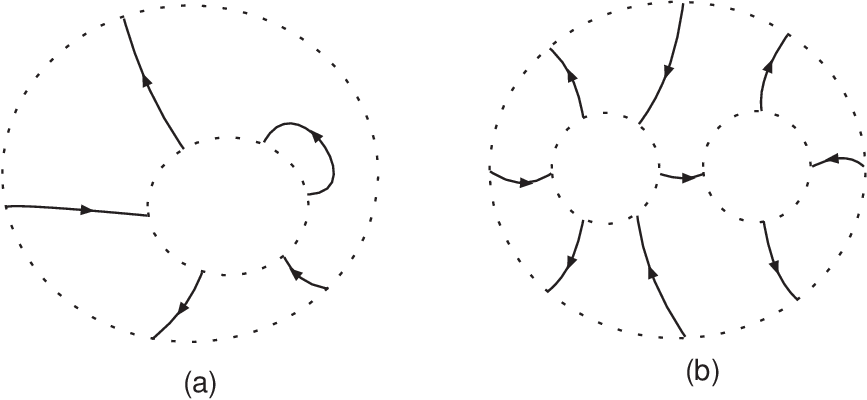}%
\caption{Examples of basic planar diagrams}\label{Fig:bplanar}
\end{figure}
\begin{definition} A basic planar diagram is a 1-input alternating planar
diagram with a curl in it,  or a 2-input alternating planar diagram with only one interconnecting arc. A basic operator is one defined from a basic planar diagram.
A negative unary basic operator is one defined from a basic 1-input diagram where the curl completes a negative loop. A positive unary basic operator is one defined from a basic 1-input diagram where the curl completes a positive loop. A binary operator is one defined from a basic 2-input planar diagram.
\end{definition}
\begin{proposition}\label{prop:RotNumberBasic}
The rotation associated number of a planar diagram belongs to
$\frac{1}{2}\mathbb{Z}$ and the case when we have a basic planar
diagram it is given as follows:
\begin{itemize}
\item If $D$ is  a negative unary basic operator, $R_D=-\frac{1}{2}$
\item If $D$ is  a  binary basic operator, $R_D=0$
\item If $D$ is  a positive unary basic operator, $R_D=\frac{1}{2}$
\end{itemize}
\end{proposition}
\begin{proposition}\label{prop:compplanar} Any operator $D$ in an alternatively oriented planar algebra is the finite composition of basic operators.
\end{proposition}
%%%%%%%%%%%%%%%%%%%%%%%%%%%%%%%%%%%%%%%%%%%%%%%%%%%%%%%%%%%%%%%%%%%%
%%%%%%%%%%% New Section
%%%%%%%%%%%%%%%%%%%%%%%%%%%%%%%%%%%%%%%%%%%%%%%%%%%%%%%%%%%%%%%%%%%%
\section{Diagonal complexes } \label{sec:On-Diagonal}
 Once we have deleted a loop in an element of $\Kobo$, we obtain a complex
$(\Omega,d)$, which preserves some properties of the former one, but
with a change in the rotation number of the element $\sigma\{q_{\sigma}\}$, in which we have
applied the delooping. In fact, the smoothing has been replaced in
the complex by a couple whose rotation number has changed either by -1 or
by +1. This shift in the rotation number could be
even greater if we continue removing loops in the same smoothing. So it would be a good idea to define a concept that states a relation between the rotation number of $\sigma$ and its grading shift $q_{\sigma}$.
\begin{definition} Let $(\Omega,d)$ be a class-representative of $ \Koboh $, and let
$\sigma_i\{q_{i}\}$ be a shifted degree object in
$\Omega^r$, then its degree-shifted rotation number is
$R(\sigma_i\{q_{i}\})=R(\sigma_i)+q_{i}$
\end{definition}

 \begin{definition} A  diagonal complex is a degree-preserving differential
chain complex $(\Omega,d)$
\[\xymatrix{  \cdots \Omega^r
\ar[r]^{d^r} & \Omega^{r+1}
  \cdots } \]
 in $\Kobo$,  satisfying that for each homological degree $r$ and each shifted degree object $\sigma_i\{q_{i}\}$ in $\Omega^r$, we have that $2r-R(\sigma_i\{q_{i}\})=C_{\Omega}$, where $C_{\Omega}$ is a constant that we call rotation constant of $(\Omega,d)$.
\end{definition}
 Here we have some examples of diagonal complexes in $\Kobo$.
\begin{example}\label{Ex:diagonal} As in \cite{Bar2}, a dotted line represent a dotted curtain, and $\HSaddleSymbol$ stands for the saddle $\smoothing \longrightarrow \hsmoothing$
\begin{enumerate}
\item \[\Omega_1= \begin{diagram}
\node{\begin{tabular}{c}\includegraphics[scale=.5]%
{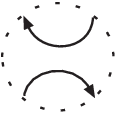}\end{tabular}\{-2\}}\arrow{e,t}{\ISaddleSymbol} \node{\begin{tabular}{c}\includegraphics[scale=.5]%
{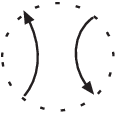}\end{tabular} \{-1\}} \end{diagram}.\] This is the Khovanov homology of the negative crossing $\undercrossing$, now with orientation in the smoothings. Remember that the first term has homological degree -1. In this example the rotation number in the first term is $-\frac{1}{2}$ and in the second term it is $\frac{1}{2}$. Observe that in each case, the difference between 2 times the homological degree $r$ and the shifted rotation number is $\frac{1}{2}$.
\item \begin{figure}[hbt] \centering
\includegraphics[scale=0.6]%
{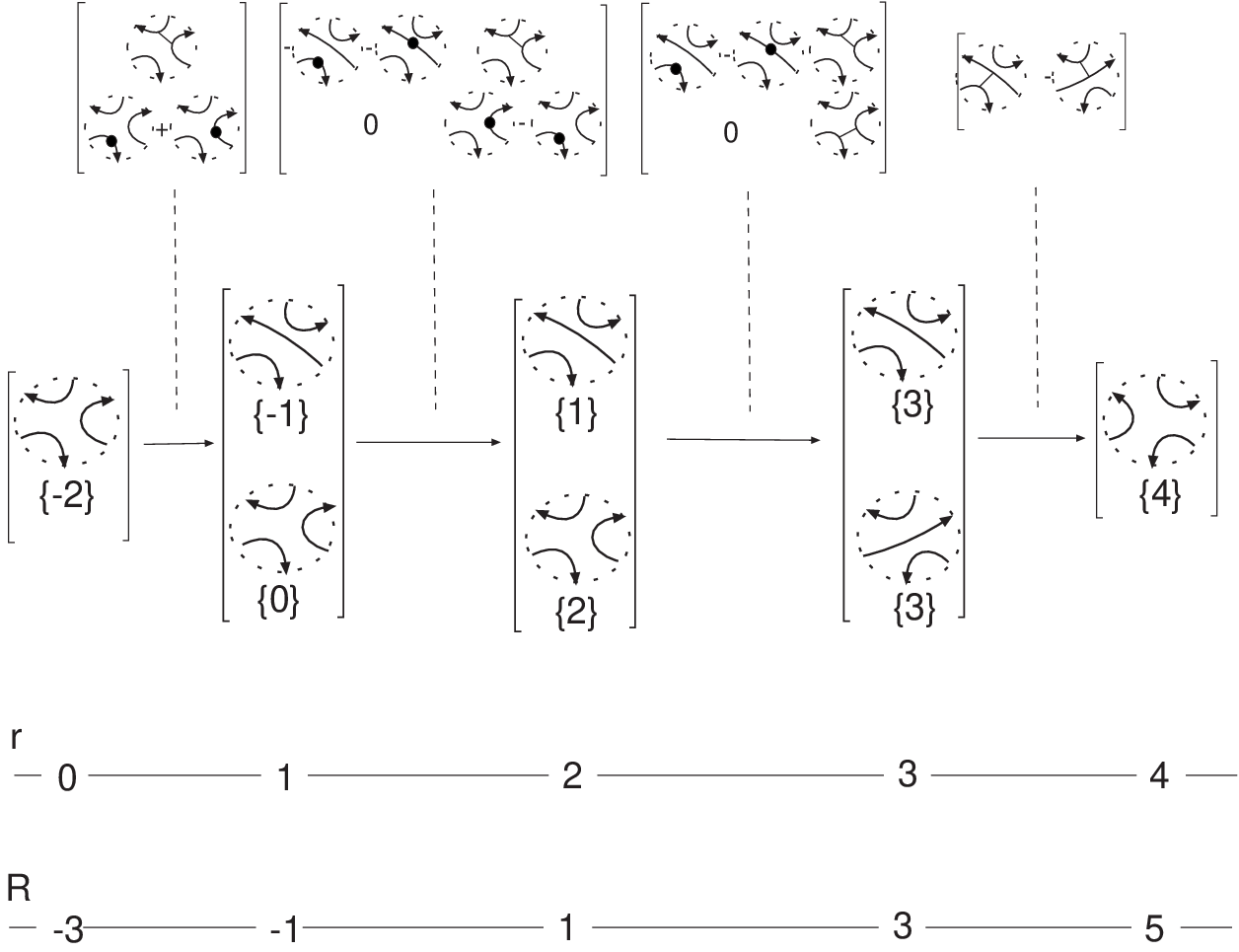}
\caption{A diagonal complex.
}\label{Fig:diagonal}
\end{figure}

 In Figure \ref{Fig:diagonal}, the number below each smoothing is the grading shift of the smoothing. The upper line below the complex represents the homological degree $r$, and the lower one represents the degree-shifted rotation number.  For instance, the rotation number in the first smoothing with homological degree 1 has rotation number 0 and a grading shift by -1. In the second smoothing of the same vector, the rotation number is -1 and its grading shift is 0, so both term has the same degree-shifted rotation number. We see in this example, that for each $r$ we have that $2r-R=3$, so this is a diagonal complex.
\end{enumerate}
\end{example}
Now, we can establish a parallel between what we did with alternating elements in $\calMko$ and diagonal complexes in $\Kobo$ in such a way that we can obtain similar results as those obtained in section 4 of \cite{Bur}.

%%%%%%%%%%%%%%%%%%%% New Subsection

\subsection{Applying unary operators}
The reduced complexes in $\Kom(\Mat(\Cobo))$ can be inserted in appropriate unary basic planar diagrams, and then apply delooping and gaussian elimination to obtain again a reduced complex in $\Kobo$. This process can be summarized in the following steps:
\begin{enumerate}
\item placing of the complex in the corresponding input disc of the $d$-input planar arc diagram by using equations (\ref{eq:KobPA}),
\item removing the loops obtained and replacing each of them by a copy of $\emptyset\{+1\}\oplus\emptyset\{-1\}$, and
\item applying gaussian elimination, and removing in this way each invertible differential in the complex.
\end{enumerate}

 \begin{definition}Let $(\Omega,d)$ be a chain complex in $\Kom(\Mat(\Cobo(k)))$, then a partial closure of $(\Omega,d)$ is a chain complex of the form $D_l\circ \cdots \circ D_1(\Omega)$ where $0\leq l<k$ and every $D_i$ ($1\leq i\leq l$) is a unary basic operator
\end{definition}
\parpic[r]{\includegraphics[scale=.6]%
{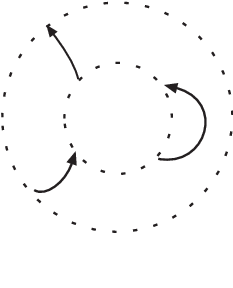}}We have diagonal complexes whose partial closures are again diagonal complexes. For instance, embedding $\Omega_1$ of the example \ref{Ex:diagonal} in a unary basic planar diagram $U_1$ as the one on the right which has an associated rotation number $R_{U_1} = \frac{1}{2}$, produces the chain complex.

\begin{eqnarray*}U_1(\Omega_1) &=& \begin{diagram}
\node{\left[\begin{tabular}{c}\includegraphics[scale=.5]%
{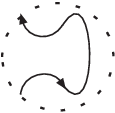}\end{tabular}\{-2\}\right]}\arrow{e,t}{\left[\ISaddleSymbol\right]} \node{\left[\begin{tabular}{c}\includegraphics[scale=.5]%
{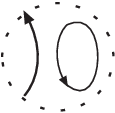}\end{tabular} \{-1\}\right]} \end{diagram}\\
& & \vspace{.5cm}\\
&\sim & \begin{diagram}
\node{\left[\begin{tabular}{c}\includegraphics[scale=.5]%
{figs/alternating121.eps}\end{tabular}\{-2\}\right]}\arrow{e,t}{\left[                                                               \begin{array}{c}
\includegraphics[scale=.3]%
{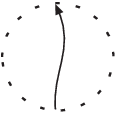} \\
\includegraphics[scale=.3]%
{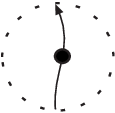} \\
\end{array}
\right]
} \node{\left[\begin{array}{l}\begin{tabular}{c}\includegraphics[scale=.5]%
{figs/alternating131.eps}\end{tabular}\{-2\}\\
\begin{tabular}{c}\includegraphics[scale=.5]%
{figs/alternating131.eps}\end{tabular}\{0\}\end{array}\right]} \end{diagram}
\end{eqnarray*}
The last complex is the result of applying del loping. Applying now gaussian elimination we obtain a homotopy equivalent complex \begin{eqnarray*}U_1(\Omega_1) &\sim& \begin{diagram}
\node{0}\arrow{e,t}{\left[0\right]} \node{\left[\begin{tabular}{c}\includegraphics[scale=.5]%
{figs/alternating131.eps}\end{tabular} \{0\}\right]} \end{diagram}
\end{eqnarray*}
which is also a diagonal complex, but now with rotation constant zero.

\begin{definition}
Let $(\Omega,d)$ be a bounded diagonal complex in $\Kobo$ with rotation constant $C_R$. We say that $(\Omega,d)$ is {\it coherently diagonal} if for any appropriated unary operator with associated rotation number $R_U$, the closure $U(\Omega,d)$ has a reduced form which is a diagonal complex with rotation constant $C_R-R_U$.
\end{definition}
We denote as $\calD(k)$ the collection of all coherently diagonal complexes in $\Kom(\Mat(\Cobo(k)))$, and as usual, we write $\calD$ to denote $\bigcup_k \calD(k)$. It is easy to prove that any coherently diagonal complex satisfies that:
\begin{enumerate}
\item after delooping any of the positive loops obtained in any of its partial closure, by using gaussian elimination, the negative shifted-degree term can be eliminated.
\item after delooping any of the negative loops obtained in any of its partial closure, by using lemma gaussian elimination, the positive shifted-degree term can be eliminated.
\end{enumerate}

\parpic[r]{\includegraphics[scale=.8]%
{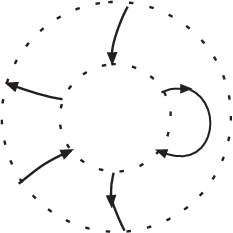}}Since the computation of any other of its partial closures produces other diagonal complex, the complex $\Omega_1$ of the example \ref{Ex:diagonal} is an element of $\calD(2)$. Another example of coherently diagonal complex is the complex $\Omega_2$ of the same example. This last complex has $C_R=3$. All of its partial closures $U(\Omega_2)$ are diagonal complexes with rotation constant given by $C_R-R_U$. Here, we only calculate the one produced by inserting the element in the closure disc $U$, with $R_U=-\frac{1}{2}$, that appears on the right. It will be easy for the reader to compute the other partial closures. Inserting $\Omega_2$ in $U$ produces the complex of Figure \ref{Fig:DiagonalWithLoop}, \begin{figure}[hbt] \centering
\includegraphics[scale=0.6]%
{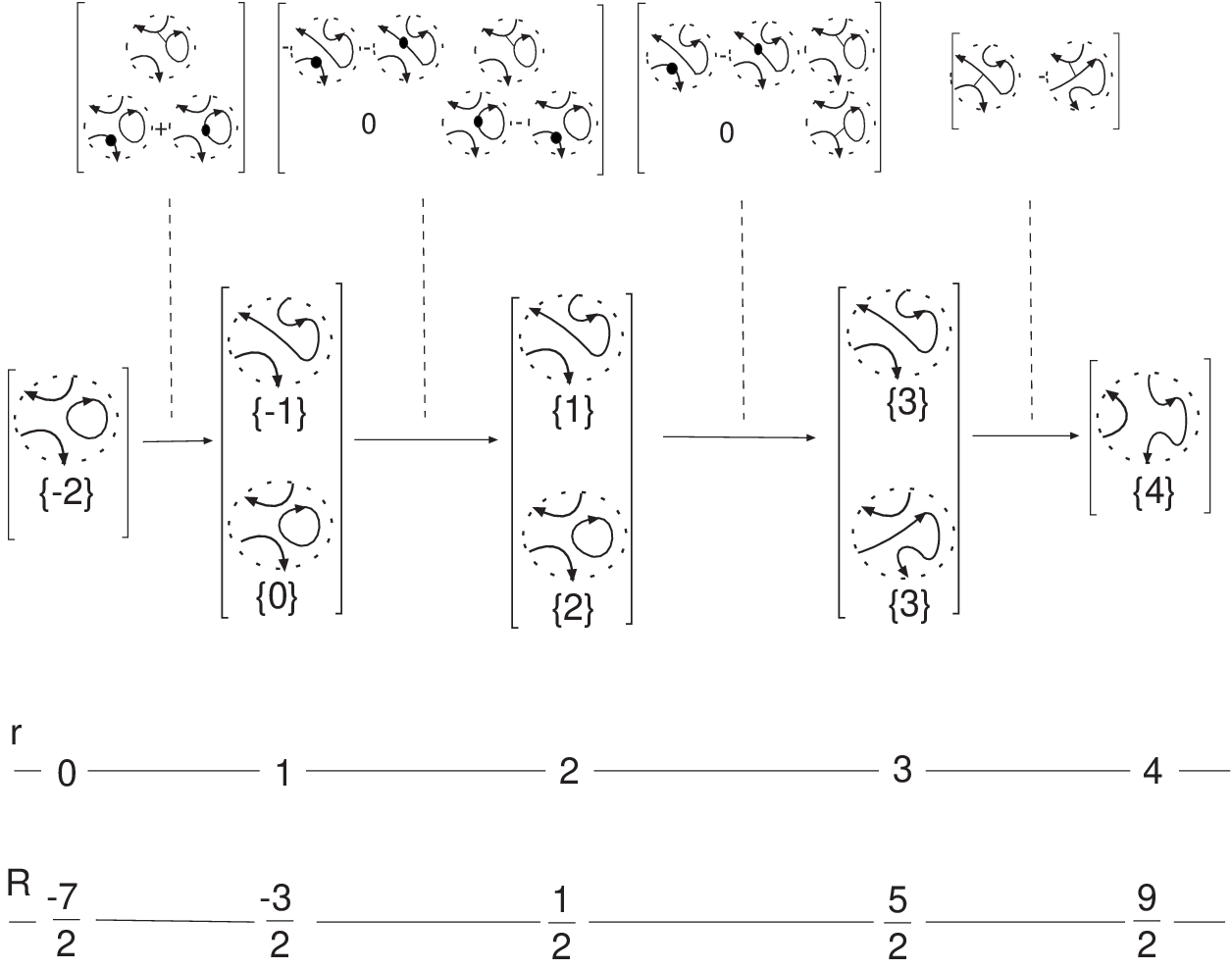}
\caption{A diagonal complex inserted in a negative unary basic diagram $U$.
}\label{Fig:DiagonalWithLoop}
\end{figure}
which is also a diagonal complex, but with a loop in some of its smoothings. Observe that the rotation number of the smoothings have decreased in $\frac{1}{2}$ after having been inserted in a negative unary basic diagram.\\
After applying delooping and gaussian elimination, we obtain the complex in Figure \ref{Fig:2Diagonal} which is also a diagonal complex, but now with rotation constant $\frac{7}{2}$.
\begin{figure}[hbt] \centering
\includegraphics[scale=0.6]%
{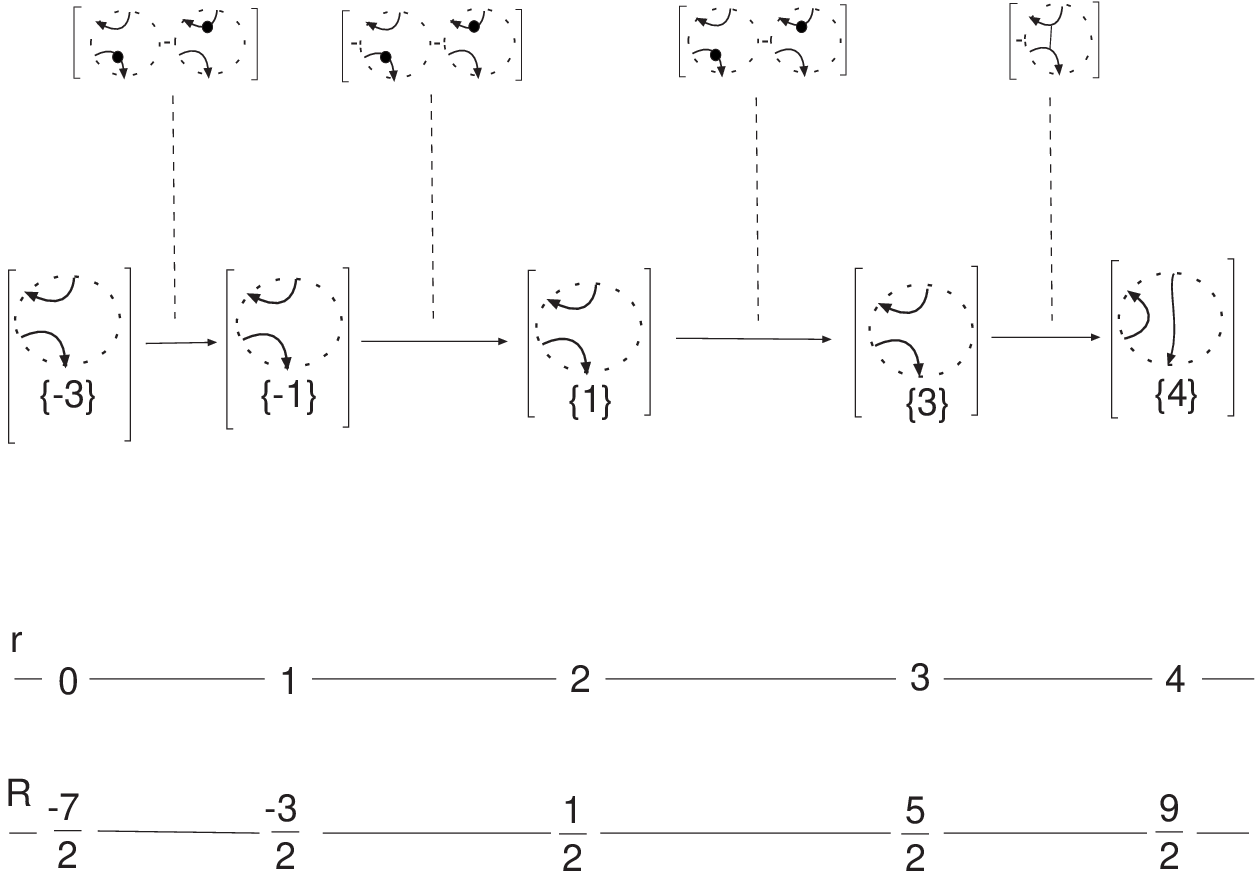}
\caption{A partial closure of a coherently diagonal complex is also a diagonal complex.
}\label{Fig:2Diagonal}
\end{figure}

%%%%%%%%%%%%%%%%%%%%%%%%%%% New Subsection %%%%%%%%%%%%

\subsection{Applying binary operators}
\begin{proposition}\label{prop:ComBinaryOperator}
If $D$ is an appropriate binary basic operator and $(\Psi,e),(\Phi,f)$ are diagonal complexes in  $\Kobo$ with rotation constants $C_{\Psi}$ and $C_{\Phi}$ respectively, then $D(\Psi,\Phi)$ is a diagonal complex with rotation constant $C_{\Psi}+C_{\Phi}$.
\end{proposition}
\begin{proof} Inserting $(\Psi,e)$ and $(\Phi,f)$ in the disc $D$ produces the complex $(\Omega,d)= D(\Psi,\Phi)$, which
by equation (\ref{eq:KobPA}) satisfies
\begin{equation}\label{eq:objbasic2}\Omega^r=
\bigoplus_{r=s+t}D(\Psi^{s},\Phi^{t})\end{equation}
and
\begin{equation}\label{eq:diffbasic2}d_{|D(\Psi^{s},\Phi^{t})}=
D(e,I_{\Phi^{t}})+(-1)^{s}D(I_{\Psi^{s}},f)\end{equation}

If $\psi\{q_{\psi}\}$ and $\phi\{q_{\phi}\}$ are
respectively elements in the vectors $\Psi^s$ and $\Phi^t$, so
by equation (\ref{eq:objbasic2}) the elements in the vector
$\Omega^r$ are of the form $D(\psi,\phi)\{q_{\psi}+q_{\phi}
\}$. As $\psi$ and $\phi$ are smoothings with no loops, the same
we have for $D(\psi,\phi)$ and by using propositions 3.7 and 3.10 in \cite{Bur}, we
obtain
\[R(D(\psi,\phi))+q_{\psi}+q_{\phi}=R(\psi)+R(\phi)+q_{\psi}+q_{\phi}.\]
Therefore, the homological degree $r$ is given by $s-(R(\psi)+q_{\psi})+t-(R(\phi)+q_{\phi})=C_{\Psi}+C_{\Phi}$ \qed
\end{proof}

\begin{proposition}\label{prop:ComplexUnaryToBinary} Let $(\Psi,e)$ and $(\Phi,e)$ complex in $\calD$ with rotation constant $C_{\Psi}$ and $C_{\Phi}$ respectively, and let $D$ be a binary basic planar operator in which $D(\Psi,\Phi)$ is well defined. For each partial closure $C(D(\Psi,\Phi))$, there exists an operator $D'$ defined on a diagram without curls and chain complexes $\Psi',\Phi'$ in $\calD$ such that \[C(D(\Psi,\Phi))=D'(\Psi',\Phi')\]. \end{proposition}
\begin{proof} The  proof similar to the proof of proposition 4.7 in \cite{Bur} \qed
\end{proof}
\begin{proposition}\label{prop:ComplexSmoothingDiag}Let $\sigma$ and $\tau$ be smoothings, and let $D$ be a suitable binary planar operator defined from a no-curl planar arc diagram with output disc $D_0$, input discs $D_1,D_2$, associated rotation constant $R_D$ and with at least one boundary arc ending in $D_1$, then there exists a closure operator $C$ and a unary operator $D'$ defined from a no-curl planar arc diagram such that $D(\sigma, \tau)=D'(C(\sigma))$. Moreover, if $(\Omega,d) \in \calD$ has rotation constant $C_{\Omega}$, then $D(\Omega,\tau)$ is a diagonal complex with rotation constant $C_{\Omega}-R(\tau)-R_D$. If  $\left[\sigma_j\right]_j$ is a vector in which each smoothing $sigma_j$ has the same rotation number $R$, then $D(\Omega',\left[\sigma_j\right]_j)$ with rotation constant $C_{\Omega}-R-R_D$
\end{proposition}
\begin{proof} The prove that  there exists a closure operator $C$ and a unary operator $D'$ defined from a no-curl planar arc diagram such that $D(\sigma, \tau)=D'(C(\sigma))$, we reason as in the proof of proposition 4.8 in \cite{Bur}. To prove that the rotation constant of $D(\Omega,\tau)$ is $C_{\Omega}-R(\tau)-R_D$, we observe that for each smoothing $\sigma\{q_{\sigma}\}$ in $\Omega$ the shifted rotation number satisfies $R(D(\sigma\{q_{\sigma}\,\tau))= R_D+R(\sigma\{q_{\sigma}\}+R(\tau)= R_D+2r -C_{\Omega}+R(\tau)$. Therefore, $2r-R(D(\sigma\{q_{\sigma}\,\tau))=C_{\Omega}-R(\tau)-R_D$.\\
\indent  The complex $D(\Omega',\left[\sigma_j\right]_j)$ is the direct sum $\oplus_j \left[ D\left(\Omega', \sigma_j\right)\right]$. Thus, the last part of the proposition follows from the observation that each of its direct summands $D\left(\Omega', \sigma_j\right)$ is a coherently diagonal complex with rotation constant $C-R-R_D$. \qed
\end{proof}

%%%%%%%%%%%%%%%%%%%%%%%%%%%%%%%%%%%%%%%%%%%%%%%%%%%%%%%%%%%%%%
%%%%%%%%%%%%%%%%%%%%% New Section
%%%%%%%%%%%%%%%%%%%%%%%%%%%%%%%%%%%%%%%%%%%%%%%%%%%%%%%%%%%%%%%

\section{Proof of Main Theorem}\label{sec:Theorem1} We are ready to prove our main result

\begin{proof}(Of  \textbf{Main Theorem})
Assume that $(\Phi,f)$ is bounded in $[t_0,t_M]$. Let $t_M-t_0=N-1$, we apply induction on $N$. For the case N=1, the result is obvious by proposition \ref{prop:ComplexSmoothingDiag}.\\
 \indent Assume that the statement is valid for any diagonal complex with numeration $g:S\longrightarrow \{1,...,N-1\}$.   $(\Omega,d)=D(\Psi,\Phi)$, let $(\Phi',f')$ be the complex resulting from eliminating  in $(\Phi,f)$, the last smoothing $\phi_N$ and every cobordism that have $\phi_N$ as the image. It will be easy for the reader to prove that $(\Phi',f')$ is in fact a chain complex. By proposition \ref{lem:TriangulOfBlocks} (also observe Remark \ref{re:ComplexSmoothingComposition}), we have that the complex $D(\Psi,\Phi)$ is formed by segments of the form

 \begin{equation}\label{eq:InductionTriangularBlockSegment}
  \xymatrix@C=3cm{
    \cdots\
    {\begin{bmatrix}\Omega_{(1)}^{r-1} \\ \Omega_N^{s-1}\end{bmatrix}}
    \ar[r]^{\begin{pmatrix}
      d_{(1)}^{r-1} & 0 \\ \rho_{(1)}^{r-1} & d_N^{s-1}
    \end{pmatrix}} &
    {\begin{bmatrix}\Omega_{(1)}^r \\ \Omega_N^s\end{bmatrix}}
    \ar[r]^{\begin{pmatrix}
      d_{(1)}^r & 0 \\ \rho_{(1)}^r & d_N^s
    \end{pmatrix}} &
    {\begin{bmatrix}\Omega_{(1)}^{r+1} \\ \Omega_N^{s+1}\end{bmatrix}}
    \ar[r]^{\begin{pmatrix}
      d_{(1)}^{r+1} & 0 \\ \rho_{(1)}^{r+1} & d_N^{s+1}
    \end{pmatrix}} &
    {\begin{bmatrix}\Omega_{(1)}^{r+2} \\ \Omega_N^{s+2}\end{bmatrix}} \  \cdots
  }
\end{equation}
Here, $(\Omega_{(1)},d_{(1)})=D(\Psi,\Phi')$, $d_N^s=D(e^s,I_{\Phi_N})$, $s=r-t_M$, and $\rho^r=(-1)^sD(I_{\Psi^s},f_{N-1})$, where $f_{N-1}$ is the component of $f$ that has $\phi_N$ as the image. A $0$ here actually represents a vector of zero morphisms\\
\indent By the induction hypothesis, $D(\Phi,\Psi')$ is a coherently alternating complex, so it is possible to carry out delooping and gauss eliminations in  $D(\Phi,\Psi')$ and obtain a reduced diagonal complex $(\Omega_{(0)},d_{(0)})$ which is a diagonal complex with rotation constant $C_{\Psi}+C_{\Phi}-R_D$. Furthermore, anytime before applying Gaussian elimination we have had complex segments of the form:

 \begin{equation}\label{eq:TriangularBlockSegment}
  \xymatrix@C=3cm{
    \cdots\
    {\begin{bmatrix}C \\ \Omega_N^{s-1}\end{bmatrix}}
    \ar[r]^{\begin{pmatrix}
      \alpha & 0 \\ \beta & 0 \\ \rho_1 & d_N^{s-1}
    \end{pmatrix}} &
    {\begin{bmatrix}b_1\\ D \\ \Omega_N^s\end{bmatrix}}
    \ar[r]^{\begin{pmatrix}
     \phi & \delta & 0 \\  \gamma & \epsilon & 0 \\ \rho_2 & \rho_3 & d_N^s
    \end{pmatrix}} &
    {\begin{bmatrix}b_2 \\ E\\  \Omega_N^{s+1}\end{bmatrix}}
    \ar[r]^{\begin{pmatrix}
      \mu & \nu & 0 \\ \rho_4 & \rho_5 & d_N^{s+1}
    \end{pmatrix}} &
    {\begin{bmatrix}D \\ \Omega_N^{s+2}\end{bmatrix}} \  \cdots
  }
\end{equation}

where $\rho_1,\rho_2,\rho_3,\rho_4$ and $\rho_5$ denote matrices of appropriate dimensions that have been obtained in intermediate steps of the process in the places where the $\rho_{(1)}^{r}$ were located at the beginning. After Applying Gauss elimination the resulting four term complex segment is:
 \begin{equation}\label{eq:TriangularBlockSegment1}
  \xymatrix@C=3cm{
    \cdots\
    {\begin{bmatrix}C \\ \Omega_N^{s-1}\end{bmatrix}}
    \ar[r]^{\begin{pmatrix}
       \beta & 0 \\ \rho_1 & d_N^{s-1}
    \end{pmatrix}} &
    {\begin{bmatrix} D \\ \Omega_N^s\end{bmatrix}}
    \ar[r]^{\begin{pmatrix}
 \epsilon - \gamma\phi^{-1}\delta & 0 \\  \rho_3 - \rho_2\phi^{-1}\delta& d_N^s
    \end{pmatrix}} &
    {\begin{bmatrix}E \\ \Omega_N^{s+1}\end{bmatrix}}
    \ar[r]^{\begin{pmatrix}
     & \nu & 0 \\ & \rho_5 & d_N^{s+1}
    \end{pmatrix}} &
    {\begin{bmatrix}D \\ \Omega_N^{s+2}\end{bmatrix}} \  \cdots
  }
\end{equation}
Thus, by applying delooping and gaussian elimination we do not change the configuration of the right lower block in the matrices of equation (\ref{eq:InductionTriangularBlockSegment}). Thus, the complex $D(\Phi,\Psi)$ is homotopy equivalent to a complex with segments

\begin{equation}\label{eq:InductionTriangularBlockSegment2}
 \xymatrix@C=3cm{
    \cdots\
    {\begin{bmatrix}\Omega_{(0)}^{r-1} \\ \Omega_N^{s-1}\end{bmatrix}}
    \ar[r]^{\begin{pmatrix}
      d_{(0)}^{r-1} & 0 \\ \rho_{(0)}^{r-1} & d_N^{s-1}
    \end{pmatrix}} &
    {\begin{bmatrix}\Omega_{(0)}^r \\ \Omega_N^s\end{bmatrix}}
    \ar[r]^{\begin{pmatrix}
      d_{(0)}^r & 0 \\ \rho_{(0)}^r & d_N^s
    \end{pmatrix}} &
    {\begin{bmatrix}\Omega_{(0)}^{r+1} \\ \Omega_N^{s+1}\end{bmatrix}}
    \ar[r]^{\begin{pmatrix}
      d_{(0)}^{r+1} & 0 \\ \rho_{(0)}^{r+1} & d_N^{s+1}
    \end{pmatrix}} &
    {\begin{bmatrix}\Omega_{(0)}^{r+2} \\ \Omega_N^{s+2}\end{bmatrix}} \  \cdots
  }
\end{equation}

that  have loops only in the column blocks ${\Omega_N^s}$. Moreover, this complex is homotopy equivalent to the complex \\
\begin{equation}\label{eq:InductionTriangularBlockSegment3}
 \xymatrix@C=3cm{
    \cdots\
    {\begin{bmatrix} \Omega_N^{s-1}\\ \Omega_{(0)}^{r-1} \end{bmatrix}}
    \ar[r]^{\begin{pmatrix}
      d_N^{s-1} &  \rho_{(0)}^{r-1} \\ 0 & d_{(0)}^{r-1}
    \end{pmatrix}} &
    {\begin{bmatrix} \Omega_N^s \\ \Omega_{(0)}^r \end{bmatrix}}
    \ar[r]^{\begin{pmatrix}
   d_N^s   & \rho_{(0)}^r  \\ 0 &  d_{(0)}^r
    \end{pmatrix}} &
    {\begin{bmatrix}  \Omega_N^{s+1}\\ \Omega_{(0)}^{r+1}\end{bmatrix}}
    \ar[r]^{\begin{pmatrix}
     d_N^{s+1} & \rho_{(0)}^{r+1} \\ 0 &  d_{(0)}^{r+1}
    \end{pmatrix}} &
    {\begin{bmatrix}\Omega_{(0)}^{r+2} \\ \Omega_N^{s+2}\end{bmatrix}} \  \cdots
  }
\end{equation}
It is not difficult to show that  by applying delooping and Gaussian elimination we do not change either the configuration of the right lower block in the matrices of equation (\ref{eq:InductionTriangularBlockSegment3}). Furthermore, according to proposition \ref{prop:ComplexSmoothingDiag}  the chain complex 
\[D(\Psi,\phi_N)= \xymatrix@C=1cm{
    \cdots\
     {\Omega_N^{s-1}}
    \ar[r]^{d_N^{s-1}}
     &
     {\Omega_N^s}
    \ar[r]^{d_N^s} &
    {\Omega_N^{s+1}}
     \  \cdots
  }\]
   is a coherently diagonal complex with rotation constant $C_{\Psi}-\overline{R}(\phi_N)-R_D$, then for each homological degree $s$ of $(\Psi,e)$, and each smoothing $\psi$ in $\Psi^s$, we have that $2s-R(D(\psi,\phi_N))=C_{\Psi}-R(\phi_N)-R_D$. Adding $2t_M$ to each side of this last equation we obtain $2r-R(D(\psi,\phi_N))=C_{\Psi}+C_{\Phi}-R_D$. That proves that we can obtain a reduced diagonal complex from $(\Omega,d)$.
\qed
\end{proof}

\end{document}